\documentclass{article}
\usepackage{amsmath}
\usepackage{amsfonts}
\usepackage{amssymb}
\usepackage{graphicx}

\newtheorem{theorem}{Theorem}

\newtheorem{corollary}[theorem]{Corollary}

\newtheorem{definition}[theorem]{Definition}
\newtheorem{example}[theorem]{Example}

\newtheorem{lemma}[theorem]{Lemma}

\newenvironment{proof}[1][Proof]{\textbf{#1.} }{\ \rule{0.5em}{0.5em}}
\input{tcilatex}
\begin{document}

\title{Integer Powers of Certain Complex Pentadiagonal $2-$Toeplitz Matrices}
\author{Hatice K\"{u}bra Duru\thanks{%
hkduru@selcuk.edu.tr} and Durmu\c{s} Bozkurt\thanks{%
dbozkurt@selcuk.edu.tr} \\
Selcuk University, Science Faculty Department of Mathematics,\\
Turkey}
\maketitle

\begin{abstract}
In this study, we get a general expression for the entries of the $s$th
power of even order pentadiagonal $2$-Toeplitz matrices.
\end{abstract}

\section{Introduction}

Gover \cite{Gover} gained eigenvalues and eigenvectors of a tridiagonal
2-Toeplitz matrix in terms of the Chebyshev polynomials. Hadj and Elouafi 
\cite{Ahmed} obtained the general expression of the characteristic
polynomial, determinant and eigenvectors for pentadiagonal matrices. Rimas 
\cite{Rimas} offered general expression for the entries of the power of
tridiagonal 2-Toeplitz matrix, in terms of the Chebyshev polynomials of the
second kind. Obvious formulas for the determinants of a band symmetric
Toeplitz matrix is restated in \cite{Elouafi}. \'{A}lvarez-Nodarse et al. 
\cite{Petronilho} gained the general expressions for the eigenvalues,
eigenvectors and the spectral measure of 2 and 3-Toeplitz matrices. The
powers of even order symmetric pentadiagonal matrices are calculated in \cite%
{Bozkurt}. \"{O}tele\c{s} and Akbulak \cite{Akbulak} viewed powers of
tridiagonal matrices. Wu \cite{Wu} calculated the powers of Toeplitz
Matrices. The powers of complex pentadiagonal Toeplitz matrices are computed
in \cite{Duru}.

This paper is organized as follows: the first section, motivated by \cite%
{Ahmed}, we apply $K_{n}$ pentadiagonal $2$-Toeplitz matrix to the
characteristic polynomial and eigenvectors of this matrix in given \cite%
{Ahmed}. In Section $2$, we obtain the eigenvalues and eigenvectors of $%
K_{n} $ pentadiagonal $2$-Toeplitz matrix. In Section $3$, the $s$th power
of pentadiagonal $2$-Toeplitz matrix we will get by using the expression $%
K_{n}^{s}=L_{n}J_{n}^{s}L_{n}^{-1}$ \cite{P.Horn}, where $J_{n}$ is the
Jordan's form of $K_{n}$ and $L_{n}$ is the transforming matrix. In Section $%
4$, some numerical examples are given.

Consider the polynomial sequence $\left\{A_{i}\right\}_{i\geq 0}$ and $%
\left\{B_{i}\right\}_{i\geq 0}$ characterized by a three-term recurrence
relation

\begin{equation}
\left. 
\begin{array}{lll}
xA_{0}(x) & = & a_{1}A_{0}(x)+b_{1}A_{2}(x) \\ 
xA_{1}(x) & = & a_{2}A_{1}(x)+b_{2}A_{3}(x) \\ 
xA_{i-1}(x) & = & c_{1}A_{i-3}(x)+a_{1}A_{i-1}(x)+b_{1}A_{i+1}(x) \; for
\,i\geq 3\, and \,i=2t+1\,(t\in\mathbb{N}) \\ 
xA_{i-1}(x) & = & c_{2}A_{i-3}(x)+a_{2}A_{i-1}(x)+b_{2}A_{i+1}(x) \; for
\,i\geq 3\, and \,i=2t \,(t\in\mathbb{N})%
\end{array}
\right.  \label{1}
\end{equation}
with initial conditions $A_{0}(x)=0$ and $A_{1}(x)=1$, and

\begin{equation}
\left. 
\begin{array}{lll}
xB_{0}(x) & = & a_{1}B_{0}(x)+b_{1}B_{2}(x) \\ 
xB_{1}(x) & = & a_{2}B_{1}(x)+b_{2}B_{3}(x) \\ 
xB_{i-1}(x) & = & c_{1}B_{i-3}(x)+a_{1}B_{i-1}(x)+b_{1}B_{i+1}(x)\;for\,i%
\geq 3\,and\,i=2t+1\,(t\in \mathbb{N}) \\ 
xB_{i-1}(x) & = & c_{2}B_{i-3}(x)+a_{2}B_{i-1}(x)+b_{2}B_{i+1}(x)\;for\,i%
\geq 3\,and\,i=2t\,(t\in \mathbb{N})%
\end{array}%
\right.  \label{2}
\end{equation}%
with initial conditions $B_{0}(x)=1$ and $B_{1}(x)=0$, here $a_{1},a_{2}\in 
\mathbb{C}$ and $b_{1},b_{2},c_{1},c_{2}\in \mathbb{C}\setminus \left\{
0\right\} $. We can write a matrix form to this three-term recurrence
relations

\begin{equation}
\left. 
\begin{array}{lll}
xA_{n-1}(x) & = & K_{n}A_{n-1}(x)+A_{n}(x)d_{n-1}+A_{n+1}(x)d_{n} \\ 
xB_{n-1}(x) & = & K_{n}B_{n-1}(x)+B_{n}(x)d_{n-1}+B_{n+1}(x)d_{n}%
\end{array}%
\right.  \label{3}
\end{equation}%
where $A_{n-1}(x)=\left[ A_{0}(x),A_{1}(x),A_{2}(x),\ldots ,A_{n-1}(x)\right]
^{T}$, \newline
$B_{n-1}(x)=\left[ B_{0}(x),B_{1}(x),B_{2}(x),\ldots ,B_{n-1}(x)\right] ^{T}$%
, 
\begin{equation*}
\left. 
\begin{array}{lll}
d_{n-1} & = & \left[ 0,0,0,\ldots ,0,b_{2},0\right] ^{T} \\ 
d_{n} & = & \left[ 0,0,0,\ldots ,0,0,b_{1}\right] ^{T}%
\end{array}%
\right\} \ (n=2t+1,\,t\in \mathbb{N})
\end{equation*}%
and 
\begin{equation*}
\left. 
\begin{array}{lll}
d_{n-1} & = & \left[ 0,0,0,\ldots ,0,b_{1},0\right] ^{T} \\ 
d_{n} & = & \left[ 0,0,0,\ldots ,0,0,b_{2}\right] ^{T}%
\end{array}%
\right\} \ (n=2t,\,t\in \mathbb{N}).
\end{equation*}%
Let 
\begin{equation}
K_{n}=\left[ 
\begin{array}{cccccccccc}
a_{1} & 0 & b_{1} & 0 & 0 & 0 & \cdots & 0 & 0 & 0 \\ 
0 & a_{2} & 0 & b_{2} & 0 & 0 & \cdots & 0 & 0 & 0 \\ 
c_{1} & 0 & a_{1} & 0 & b_{1} & 0 & \cdots & 0 & 0 & 0 \\ 
0 & c_{2} & 0 & a_{2} & 0 & b_{2} & \cdots & 0 & 0 & 0 \\ 
0 & 0 & c_{1} & 0 & a_{1} & 0 & \cdots & 0 & 0 & 0 \\ 
\vdots & \vdots & \vdots & \vdots & \vdots & \vdots & \ddots & \vdots & 
\vdots & \vdots \\ 
0 & 0 & 0 & 0 & 0 & 0 & \cdots & 0 & b_{2} & 0 \\ 
0 & 0 & 0 & 0 & 0 & 0 & \cdots & a_{1} & 0 & b_{1} \\ 
0 & 0 & 0 & 0 & 0 & 0 & \cdots & 0 & a_{2} & 0 \\ 
0 & 0 & 0 & 0 & 0 & 0 & \cdots & c_{1} & 0 & a_{1} \\ 
&  &  &  &  &  &  &  &  & 
\end{array}%
\right] \,for\,n=2t+1\,(t\in \mathbb{N})  \label{4}
\end{equation}%
and 
\begin{equation}
K_{n}=\left[ 
\begin{array}{cccccccccc}
a_{1} & 0 & b_{1} & 0 & 0 & 0 & \cdots & 0 & 0 & 0 \\ 
0 & a_{2} & 0 & b_{2} & 0 & 0 & \cdots & 0 & 0 & 0 \\ 
c_{1} & 0 & a_{1} & 0 & b_{1} & 0 & \cdots & 0 & 0 & 0 \\ 
0 & c_{2} & 0 & a_{2} & 0 & b_{2} & \cdots & 0 & 0 & 0 \\ 
0 & 0 & c_{1} & 0 & a_{1} & 0 & \cdots & 0 & 0 & 0 \\ 
\vdots & \vdots & \vdots & \vdots & \vdots & \vdots & \ddots & \vdots & 
\vdots & \vdots \\ 
0 & 0 & 0 & 0 & 0 & 0 & \cdots & 0 & b_{1} & 0 \\ 
0 & 0 & 0 & 0 & 0 & 0 & \cdots & a_{2} & 0 & b_{2} \\ 
0 & 0 & 0 & 0 & 0 & 0 & \cdots & 0 & a_{1} & 0 \\ 
0 & 0 & 0 & 0 & 0 & 0 & \cdots & c_{2} & 0 & a_{2} \\ 
&  &  &  &  &  &  &  &  & 
\end{array}%
\right] \,for\,n=2t\,(t\in \mathbb{N})  \label{5}
\end{equation}%
here $a_{1},a_{2}\in \mathbb{C}$ and $b_{1},b_{2},c_{1},c_{2}\in \mathbb{C}%
\setminus \left\{ 0\right\} $.

\begin{lemma}
The polynomial sequences $\left\{ A_{i}\right\} _{i\geq 0}$ and $\left\{
B_{i}\right\} _{i\geq 0}$ confirm 
\begin{equation*}
\begin{array}{lll}
deg(A_{2i+1}) & = & i\;\text{and\ the\ leading\ coefficient\ of}\;A_{2i+1}\;%
\text{is\ equal\ to}\;\frac{1}{b_{2}^{i}}, \\ 
deg(A_{2i}) & = & 0, \\ 
deg(B_{2i+1}) & = & 0, \\ 
deg(B_{2i}) & = & i\;\text{and\ the\ leading\ coefficient\ of\ }B_{2i}\;%
\text{is\ equal\ to\ }\frac{1}{b_{1}^{i}}.%
\end{array}%
\end{equation*}
\end{lemma}

\begin{proof}
Let us prove by the inductive method. For the basis step, we possess For $%
i=0:\,A_{0}\left(x\right)=0,\,B_{0}\left(x\right)=1$.\newline
For $i=1:\,A_{1}\left(x\right)=1,\,B_{1}\left(x\right)=0$.\newline
For $i=2:\,A_{2}\left(x\right)=0,\,B_{2}\left(x\right)=\frac{x-a_{1}}{b_{1}}$%
.\newline
For $i=3:\,A_{3}\left(x\right)=\frac{x-a_{2}}{b_{2}},\,B_{3}\left(x\right)=0$%
.\newline
For $i=4:\,A_{4}\left(x\right)=0,\,B_{4}\left(x\right)=\frac{%
(x-a_{1})^2-b_{1}c_{1}}{b_{1}^2}$.\newline
Assuming that (\ref{1}) and (\ref{2}) are correct for $p=k\geq 3$. We will
prove it for $k=p+1$, we obtain 
\begin{equation}
\left. 
\begin{array}{l}
xA_{p-1}(x)=c_{1}A_{p-3}(x)+a_{1}A_{p-1}(x)+b_{1}A_{p+1}(x)\;for\;p\geq
3\;and\;p=2t+1\,(t\in\mathbb{N}) \\ 
xA_{p-1}(x)=c_{2}A_{p-3}(x)+a_{2}A_{p-1}(x)+b_{2}A_{p+1}(x)\;for\;p\geq
3\;and\;p=2t \,(t\in\mathbb{N}) \\ 
xB_{p-1}(x)=c_{1}B_{p-3}(x)+a_{1}B_{p-1}(x)+b_{1}B_{p+1}(x)\;for\;p\geq
3\;and\;p=2t+1\,(t\in\mathbb{N}) \\ 
xB_{p-1}(x)=c_{2}B_{p-3}(x)+a_{2}B_{p-1}(x)+b_{2}B_{p+1}(x)\;for\;p\geq
3\;and\;p=2t \,(t\in\mathbb{N}). \\ 
\end{array}%
\right.  \label{6}
\end{equation}

If $p:=2i$, then we have 
\begin{equation*}
\left. 
\begin{array}{l}
xA_{2i-1}(x)=c_{2}A_{2i-3}(x)+a_{2}A_{2i-1}(x)+b_{2}A_{2i+1}(x) \\ 
xB_{2i-1}(x)=c_{2}B_{2i-3}(x)+a_{2}B_{2i-1}(x)+b_{2}B_{2i+1}(x). \\ 
\end{array}%
\right.
\end{equation*}%
Accordingly 
\begin{equation*}
\left. 
\begin{array}{ll}
deg(A_{2i+1}(x)) & =deg(xA_{2i-1}(x)) \\ 
deg(B_{2i+1}(x)) & =0 \\ 
& 
\end{array}%
\right.
\end{equation*}%
and the leading coefficient of 
\begin{equation}
\left. 
\begin{array}{ll}
A_{2i+1} & =\frac{1}{b_{2}}(\text{\textit{leading\ coefficient\ of\ }}%
A_{2i-1}(x)), \\ 
& =\frac{1}{b_{2}}\frac{1}{(b_{2})^{i-1}}=\frac{1}{(b_{2})^{i}}.%
\end{array}%
\right.  \label{7}
\end{equation}

If $p:=2i-1$, then we write 
\begin{equation*}
\left. 
\begin{array}{l}
xA_{2i-2}(x)=c_{1}A_{2i-4}(x)+a_{1}A_{2i-2}(x)+b_{1}A_{2i}(x) \\ 
xB_{2i-2}(x)=c_{1}B_{2i-4}(x)+a_{1}B_{2i-2}(x)+b_{1}B_{2i}(x). \\ 
\end{array}%
\right.
\end{equation*}
So, 
\begin{equation*}
\left. 
\begin{array}{ll}
deg(A_{2i}(x)) & =0 \\ 
deg(B_{2i}(x)) & =deg(xB_{2i-2}(x)) \\ 
& 
\end{array}
\right.
\end{equation*}
and the leading coefficient of

\begin{equation}
\left. 
\begin{array}{ll}
B_{2i} & =\frac{1}{b_{1}}(\text{\textit{leading\ coefficient\ of\ }}%
B_{2i-2}(x)), \\ 
& =\frac{1}{b_{1}}\frac{1}{(b_{1})^{i-1}}=\frac{1}{(b_{1})^{i}}.%
\end{array}%
\right.  \label{8}
\end{equation}
\end{proof}

\begin{definition}
$K_{n}$ be $n$-square pentadiagonal $2$-Toeplitz matrix, one correlates the
sequence polynomial $P_{i}$ described by 
\begin{equation}
P_{i}=\det\left[%
\begin{array}{cc}
A_{n} & A_{i} \\ 
B_{n} & B_{i}%
\end{array}%
\right].  \label{9}
\end{equation}
\end{definition}

\begin{lemma}
Due to (\ref{1}) and (\ref{2}), we own

\begin{equation}
\left. 
\begin{array}{lll}
xP_{n-1}(x) & = & K_{n}P_{n-1}(x)+P_{n}(x)d_{n-1}+P_{n+1}(x)d_{n} \\ 
& = & K_{n}P_{n-1}(x)+P_{n+1}(x)d_{n}%
\end{array}%
\right.  \label{10}
\end{equation}%
where $P_{n-1}(x)=\left[ P_{0}(x),P_{1}(x),P_{2}(x),\ldots ,P_{n-1}(x)\right]
^{T}$, 
\begin{equation*}
\left. 
\begin{array}{lll}
d_{n-1} & = & \left[ 0,0,0,\ldots ,0,b_{2},0\right] ^{T} \\ 
d_{n} & = & \left[ 0,0,0,\ldots ,0,0,b_{1}\right] ^{T}%
\end{array}%
\right\} \ (n=2t+1,\,t\in \mathbb{N}),
\end{equation*}%
\begin{equation*}
\left. 
\begin{array}{lll}
d_{n-1} & = & \left[ 0,0,0,\ldots ,0,b_{1},0\right] ^{T} \\ 
d_{n} & = & \left[ 0,0,0,\ldots ,0,0,b_{2}\right] ^{T}%
\end{array}%
\right\} \ (n=2t,\,t\in \mathbb{N}).
\end{equation*}
\end{lemma}

\begin{lemma}
The polynomial $P_{n+1}$ is degree $n$ and the leading coefficient of $%
P_{n+1}$ is 
\begin{equation}
\begin{array}{lll}
-\frac{1}{(b_{1}b_{2})^i} & , & \;if\;n=2t\,(t\in\mathbb{N}) \\ 
\frac{1}{b_{1}^{i+1}b_{2}^i} & , & \;if\;n=2t+1\,(t\in\mathbb{N}).%
\end{array}
\label{11}
\end{equation}
\end{lemma}

\begin{proof}
If $n=2i$ 
\begin{equation*}
\begin{array}{ll}
P_{2i+1} & =\det\left[%
\begin{array}{cc}
A_{2i} & A_{2i+1} \\ 
B_{2i} & B_{2i+1}%
\end{array}%
\right]=A_{2i}B_{2i+1}-A_{2i+1}B_{2i}%
\end{array}%
\end{equation*}
and using Lemma 1 
\begin{equation*}
deg(P_{2i+1})=deg(A_{2i+1}B_{2i})
\end{equation*}
and the leading coefficient of 
\begin{equation*}
P_{2i+1}=-\frac{1}{b_{2}^i}\frac{1}{b_{1}^i}=-\frac{1}{(b_{1}b_{2})^i}.
\end{equation*}

If $n=2i+1$ 
\begin{equation*}
\begin{array}{ll}
P_{2i+1} & =\det\left[%
\begin{array}{cc}
A_{2i+1} & A_{2i+2} \\ 
B_{2i+1} & B_{2i+2}%
\end{array}%
\right]=A_{2i+1}B_{2i+2}-A_{2i+2}B_{2i+1}%
\end{array}%
\end{equation*}
and using Lemma 1 
\begin{equation*}
deg(P_{2i+2})=deg(A_{2i+1}B_{2i+2})
\end{equation*}
and the leading coefficient of 
\begin{equation*}
P_{2i+2}=\frac{1}{b_{2}^i}\frac{1}{b_{1}^{i+1}}.
\end{equation*}
\end{proof}

\begin{lemma}
If $\alpha $ is a zero of the polynomial $P_{n+1}$ then $\alpha $ is an
eigenvalues of the matrix $K_{n}$.
\end{lemma}

\begin{proof}
Let $\alpha$ is a zero of the polynomial $P_{n+1}$, from equation (\ref{10}%
), we acquire 
\begin{equation*}
K_{n}P_{n-1}(\alpha)=\alpha P_{n-1}(\alpha).
\end{equation*}
There are four cases to be noted.\newline

Case I. Suppose for $n=2t\,(t\in\mathbb{N})$ either $A_{n+1}(\alpha)\neq 0$
or $B_{n}(\alpha)\neq 0$. In that case $P_{0}(\alpha)=A_{n+1}(\alpha)$ and $%
P_{1}(\alpha)=-B_{n}(\alpha)$, then $P_{n-1}(\alpha)$ is a corresponding
non-null eigenvector of $K_{n}$, we acquire that $\alpha$ is an eigenvalue
of the matrix $K_{n}$.\newline

Case II. Suppose for $n=2t-1\,(t\in\mathbb{N})$ either $A_{n}(\alpha)\neq 0$
or $B_{n+1}(\alpha)\neq 0$.\newline
In that case $P_{0}(\alpha)=A_{n}(\alpha)$ and $P_{1}(\alpha)=-B_{n+1}(%
\alpha)$, then $P_{n-1}(\alpha)$ is a corresponding non-null eigenvector of $%
K_{n}$, we acquire that $\alpha$ is an eigenvalue of the matrix $K_{n}$.%
\newline

Case III. Suppose for $n=2t\,(t\in\mathbb{N})$, $A_{n}(\alpha)=B_{n+1}(%
\alpha)=0$ and $B_{n}(\alpha)\neq 0$. Let $F_{n-1}(\alpha)=-A_{n-1}(%
\alpha)B_{n}(\alpha)$, we own $K_{n}F_{n-1}(\alpha)=\alpha F_{n-1}(\alpha)$,
then $F_{n-1}(\alpha)$ is a corresponding non-null eigenvector of $K_{n}$,
we acquire that $\alpha$ is an eigenvalue of the matrix $K_{n}$.\newline

Case IV. Suppose for $n=2t+1\,(t\in\mathbb{N})$, $A_{n-1}(\alpha)=B_{n}(%
\alpha)=0$ and $A_{n}(\alpha)\neq 0$. Let $F_{n-1}(\alpha)=A_{n}(%
\alpha)B_{n-1}(\alpha)$, we own $K_{n}F_{n-1}(\alpha)=\alpha F_{n-1}(\alpha)$%
, then $F_{n-1}(\alpha)$ is a corresponding non-null eigenvector of $K_{n}$,
we acquire that $\alpha$ is an eigenvalue of the matrix $K_{n}$.\newline
\end{proof}

\begin{theorem}
Let $K_{n}$ be $n$-square pentadiagonal $2$-Toeplitz matrix and the
corresponding polynomial $P_{n+1}$ in the eq. (\ref{10}). Suppose that $%
P_{n+1}$ has simple zeros, the characteristic polynomial of $K_{n}$ is
completely\newline
\begin{equation}
\begin{array}{ll}
\left\vert xI_{n}-K_{n}\right\vert =(b_{1}b_{2})^{\frac{n}{2}}A_{n+1}B_{n},
& if\;n=2t\,(t\in \mathbb{N}) \\ 
\left\vert xI_{n}-K_{n}\right\vert =b_{1}^{\frac{n+1}{2}}b_{2}^{\frac{n-1}{2}%
}A_{n}B_{n+1}, & if\;n=2t+1\,(t\in \mathbb{N})%
\end{array}
\label{12}
\end{equation}%
here $I_{n}$ is the $n-$square identity matrix.
\end{theorem}

\begin{proof}
Let $\alpha _{1},\ldots ,\alpha _{n}$ the zeros of characteristic polynomial
of $K_{n}$. Lemma $5$, $\alpha _{1},\ldots ,\alpha _{n}$ are the eigenvalues
of the matrix $K_{n}$.
\end{proof}

\section{Eigenvalues and eigenvectors of $K_{n}$}

\begin{theorem}
Let $K_{n}$ be $n$-square $(n=2t,\,t\in\mathbb{N})$ pentadiagonal $2$%
-Toeplitz matrix as in (\ref{5}). Then the eigenvalues and eigenvectors of
the matrix $K_{n}$ are

\begin{equation}
\alpha _{k}=\left\{ 
\begin{array}{l}
a_{1}-2\sqrt{b_{1}c_{1}}cos\left( \frac{(k+1)\pi }{n+2}\right)
,\;(k=2t+1\,,\,t\in \mathbb{N}) \\ 
a_{2}-2\sqrt{b_{2}c_{2}}cos\left( \frac{k\pi }{n+2}\right) ,\;\ \ \ \
(k=2t\,,\,t\in \mathbb{N})%
\end{array}%
\right.  \label{13}
\end{equation}%
and 
\begin{equation}
\left[ 
\begin{array}{c}
B_{0}\left( \alpha _{j}\right) \\ 
B_{1}\left( \alpha _{j}\right) \\ 
B_{2}\left( \alpha _{j}\right) \\ 
\vdots \\ 
B_{n-2}\left( \alpha _{j}\right) \\ 
B_{n-1}\left( \alpha _{j}\right)%
\end{array}%
\right] \ \ (j=1,3,5,\ldots ,n-3,n-1);  \label{14}
\end{equation}%
and 
\begin{equation}
\left[ 
\begin{array}{c}
A_{0}\left( \alpha _{j}\right) \\ 
A_{1}\left( \alpha _{j}\right) \\ 
A_{2}\left( \alpha _{j}\right) \\ 
\vdots \\ 
A_{n-2}\left( \alpha _{j}\right) \\ 
A_{n-1}\left( \alpha _{j}\right)%
\end{array}%
\right] \ \ (j=2,4,6,\ldots ,n-2,n).  \label{15}
\end{equation}
\end{theorem}

\begin{proof}
We obtain for 
\begin{equation}
b_{2}^{\frac{n}{2}}A_{n+1}(x)=(b_{2}c_{2})^{\frac{n}{4}}U_{\frac{n}{2}%
}\left( \frac{x-a_{2}}{2\sqrt{b_{2}c_{2}}}\right)  \label{151}
\end{equation}%
and 
\begin{equation}
b_{1}^{\frac{n}{2}}B_{n}(x)=(b_{1}c_{1})^{\frac{n}{4}}U_{\frac{n}{2}}\left( 
\frac{x-a_{1}}{2\sqrt{b_{1}c_{1}}}\right)  \label{152}
\end{equation}%
from the recurrence relations (\ref{1}) and (\ref{2}), here $n=2t\,(t\in 
\mathbb{N})$ and $U_{n}(.)$ is the $n$th degree Chebyshev polynomial of the
second kind \cite{Mason}: 
\begin{equation*}
U_{n}(x)=\frac{sin((n+1)arccosx)}{sin(arccosx)}
\end{equation*}%
All the roots of $U_{n}(x)$ are included in the interval $[-1,1]$. Due to (%
\ref{12}), (\ref{151}) and (\ref{152}), we have 
\begin{equation}
\begin{array}{ccc}
\left\vert xI_{n}-K_{n}\right\vert & = & (b_{1}c_{1})^{\frac{n}{4}}U_{\frac{n%
}{2}}\left( \frac{x-a_{1}}{2\sqrt{b_{1}c_{1}}}\right) (b_{2}c_{2})^{\frac{n}{%
4}}U_{\frac{n}{2}}\left( \frac{x-a_{2}}{2\sqrt{b_{2}c_{2}}}\right) \\ 
& = & (b_{1}c_{1}b_{2}c_{2})^{\frac{n}{4}}U_{\frac{n}{2}}\left( \frac{x-a_{1}%
}{2\sqrt{b_{1}c_{1}}}\right) U_{\frac{n}{2}}\left( \frac{x-a_{2}}{2\sqrt{%
b_{2}c_{2}}}\right) .%
\end{array}
\label{153}
\end{equation}%
The eigenvalues of $K_{n}$ obtained as 
\begin{equation*}
\alpha _{k}=\left\{ 
\begin{array}{l}
a_{1}-2\sqrt{b_{1}c_{1}}cos\left( \frac{(k+1)\pi }{n+2}\right)
,\;(k=2t+1\,,\,t\in \mathbb{N}) \\ 
a_{2}-2\sqrt{b_{2}c_{2}}cos\left( \frac{k\pi }{n+2}\right) ,\;\ \ \ \
(k=2t\,,\,t\in \mathbb{N})%
\end{array}%
\right.
\end{equation*}%
from (\ref{153}). In \cite{Ahmed} Hadj and Elouafi proved eigenvectors of a
pentadiagonal matrix. Based on Theorem 6 and \cite{Ahmed}, we can calculate
eigenvectors of the matrix $K_{n}$
\end{proof}

\section{The integer powers of the matrix $K_{n}$}

Considering (\ref{14}) and (\ref{15}), we write down the transforming
matrices $L_{n}$\linebreak $(n=2t,\;t\in \mathbb{N})$ as following: 
\begin{equation*}
L_{n}=\left[ 
\begin{array}{cccc}
B_{0}\left( \alpha _{1}\right) & A_{0}\left( \alpha _{2}\right) & 
B_{0}\left( \alpha _{3}\right) & A_{0}\left( \alpha _{4}\right) \\ 
B_{1}\left( \alpha _{1}\right) & A_{1}\left( \alpha _{2}\right) & 
B_{1}\left( \alpha _{3}\right) & A_{1}\left( \alpha _{4}\right) \\ 
B_{2}\left( \alpha _{1}\right) & A_{2}\left( \alpha _{2}\right) & 
B_{2}\left( \alpha _{3}\right) & A_{2}\left( \alpha _{4}\right) \\ 
\vdots & \vdots & \vdots & \vdots \\ 
B_{n-3}\left( \alpha _{1}\right) & A_{n-3}\left( \alpha _{2}\right) & 
B_{n-3}\left( \alpha _{3}\right) & A_{n-3}\left( \alpha _{4}\right) \\ 
B_{n-2}\left( \alpha _{1}\right) & A_{n-2}\left( \alpha _{2}\right) & 
B_{n-2}\left( \alpha _{3}\right) & A_{n-2}\left( \alpha _{4}\right) \\ 
B_{n-1}\left( \alpha _{1}\right) & A_{n-1}\left( \alpha _{2}\right) & 
B_{n-1}\left( \alpha _{3}\right) & A_{n-1}\left( \alpha _{4}\right) \\ 
&  &  & 
\end{array}%
\right.
\end{equation*}
\begin{equation}
\left. 
\begin{array}{cccc}
\qquad \qquad \qquad \qquad \qquad \qquad \qquad & \cdots & B_{0}\left(
\alpha _{n-1}\right) & A_{0}\left( \alpha _{n}\right) \\ 
\qquad \qquad \qquad \qquad \qquad \qquad \qquad & \cdots & B_{1}\left(
\alpha _{n-1}\right) & A_{1}\left( \alpha _{n}\right) \\ 
\qquad \qquad \qquad \qquad \qquad \qquad \qquad & \cdots & B_{2}\left(
\alpha _{n-1}\right) & A_{2}\left( \alpha _{n}\right) \\ 
\qquad \qquad \qquad \qquad \qquad \qquad \qquad & \ddots & \vdots & \vdots
\\ 
\qquad \qquad \qquad \qquad \qquad \qquad \qquad & \cdots & B_{n-3}\left(
\alpha _{n-1}\right) & A_{n-3}\left( \alpha_{n}\right) \\ 
\qquad \qquad \qquad \qquad \qquad \qquad \qquad & \cdots & B_{n-2}\left(
\alpha _{n-1}\right) & A_{n-2}\left( \alpha_{n}\right) \\ 
\qquad \qquad \qquad \qquad \qquad \qquad \qquad & \cdots & B_{n-1}\left(
\alpha _{n-1}\right) & A_{n-1}\left( \alpha_{n}\right) \\ 
&  &  & 
\end{array}
\right]  \label{16}
\end{equation}%
Now, let us find the inverse matrix $L_{n}^{-1}$ of the matrix $L_{n}$. If
we denote $i$-th row of the inverse matrix $L_{n}^{-1}$ by $\mu _{i}$, then
we obtain 
\begin{equation}
\left[ 
\begin{array}{c}
q_{i}r_{1}^{l}B_{0}\left( \alpha _{i}\right) \\ 
q_{i}r_{1}^{l}B_{1}\left( \alpha _{i}\right) \\ 
q_{i}r_{1}^{l}B_{2}\left( \alpha _{i}\right) \\ 
q_{i}r_{1}^{l}B_{3}\left( \alpha _{i}\right) \\ 
q_{i}r_{1}^{l}B_{4}\left( \alpha _{i}\right) \\ 
\vdots \\ 
q_{i}r_{1}^{l}B_{n-4}\left( \alpha _{i}\right) \\ 
q_{i}r_{1}^{l}B_{n-3}\left( \alpha _{i}\right) \\ 
q_{i}r_{1}^{l}B_{n-2}\left( \alpha _{i}\right) \\ 
q_{i}r_{1}^{l}B_{n-1}\left( \alpha _{i}\right) \\ 
\end{array}%
\right] ^{T}(i=1,3,5,\ldots ,n-3,n-1);  \label{17}
\end{equation}%
and 
\begin{equation}
\left[ 
\begin{array}{c}
q_{i}r_{2}^{l}A_{0}\left( \alpha _{i}\right) \\ 
q_{i}r_{2}^{l}A_{1}\left( \alpha _{i}\right) \\ 
q_{i}r_{2}^{l}A_{2}\left( \alpha _{i}\right) \\ 
q_{i}r_{2}^{l}A_{3}\left( \alpha _{i}\right) \\ 
q_{i}r_{2}^{l}A_{4}\left( \alpha _{i}\right) \\ 
\vdots \\ 
q_{i}r_{2}^{l}A_{n-4}\left( \alpha _{i}\right) \\ 
q_{i}r_{2}^{l}A_{n-3}\left( \alpha _{i}\right) \\ 
q_{i}r_{2}^{l}A_{n-2}\left( \alpha _{i}\right) \\ 
q_{i}r_{2}^{l}A_{n-1}\left( \alpha _{i}\right) \\ 
\end{array}%
\right] ^{T}(i=2,4,6,\ldots ,n-2,n)  \label{18}
\end{equation}%
where $r_{1}=\sqrt{\frac{b_{1}}{c_{1}}},\;r_{2}=\sqrt{\frac{b_{2}}{c_{2}}}%
,\; $%
\begin{equation*}
l=\left\{ 
\begin{array}{ll}
j-1, & \;j=1,3,5,\ldots ,n-3,n-1 \\ 
j-2, & \;j=2,4,6,\ldots ,n-2,n%
\end{array}%
\right.
\end{equation*}
and

\begin{equation*}
q_{i}=\left\{ 
\begin{array}{ll}
\frac{4-\left( \frac{\alpha _{i}-a_{1}}{\sqrt{b_{1}c_{1}}}\right) ^{2}}{n+2},
& \;i=2t+1\,(t\in \mathbb{N}) \\ 
\frac{4-\left( \frac{\alpha _{i}-a_{2}}{\sqrt{b_{2}c_{2}}}\right) ^{2}}{n+2},
& \;i=2t\,(t\in \mathbb{N}).%
\end{array}%
\right.
\end{equation*}

Thus, we obtain 
\begin{equation*}
L_{n}^{-1}=\left[ 
\begin{array}{cccc}
q_{1}B_{0}\left( \alpha _{1}\right) & q_{1}B_{1}\left( \alpha _{1}\right) & 
q_{1}r_{1}^{2}B_{2}\left( \alpha _{1}\right) & q_{1}r_{1}^{2}B_{3}\left(
\alpha _{1}\right) \\ 
q_{2}A_{0}\left( \alpha _{2}\right) & q_{2}A_{1}\left( \alpha _{2}\right) & 
q_{2}r_{2}^{2}A_{2}\left( \alpha _{2}\right) & q_{2}r_{2}^{2}A_{3}\left(
\alpha _{2}\right) \\ 
q_{3}B_{0}\left( \alpha _{3}\right) & q_{3}B_{1}\left( \alpha _{3}\right) & 
q_{3}r_{1}^{2}B_{2}\left( \alpha _{3}\right) & q_{1}r_{1}^{2}B_{3}\left(
\alpha _{3}\right) \\ 
\vdots & \vdots & \vdots & \vdots \\ 
q_{n-1}B_{0}\left( \alpha _{n-1}\right) & q_{n-1}B_{1}\left( \alpha
_{n-1}\right) & q_{n-1}r_{1}^{2}B_{2}\left( \alpha _{n-1}\right) & 
q_{n-1}r_{1}^{2}B_{3}\left( \alpha _{n-1}\right) \\ 
q_{n}A_{0}\left( \alpha _{n}\right) & q_{n}A_{1}\left( \alpha _{n}\right) & 
q_{n}r_{2}^{2}A_{2}\left( \alpha _{n}\right) & q_{n}r_{2}^{2}A_{3}\left(
\alpha _{n}\right)%
\end{array}%
\right.
\end{equation*}%
\begin{equation}
\left. 
\begin{array}{cccc}
\qquad \qquad & \cdots & q_{1}r_{1}^{n-2}B_{n-2}\left( \alpha _{1}\right) & 
q_{1}r_{1}^{n-2}B_{n-1}\left( \alpha _{1}\right) \\ 
\qquad \qquad & \cdots & q_{2}r_{2}^{n-2}A_{n-2}\left( \alpha _{2}\right) & 
q_{2}r_{2}^{n-2}A_{n-1}\left( \alpha _{2}\right) \\ 
\qquad \qquad & \cdots & q_{3}r_{1}^{n-2}B_{n-2}\left( \alpha _{3}\right) & 
q_{3}r_{1}^{n-2}B_{n-1}\left( \alpha _{3}\right) \\ 
\qquad \qquad & \ddots & \vdots & \vdots \\ 
\qquad \qquad & \cdots & q_{n-1}r_{1}^{n-2}B_{n-2}\left( \alpha _{n-1}\right)
& q_{n-1}r_{1}^{n-2}B_{n-1}\left( \alpha _{n-1}\right) \\ 
\qquad \qquad & \cdots & q_{n}r_{2}^{n-2}A_{n-2}\left( \alpha _{n}\right) & 
q_{n}r_{2}^{n-2}A_{n-1}\left( \alpha _{n}\right)%
\end{array}%
\right] .  \label{19}
\end{equation}%
We write the $s$th powers of the matrix $K_{n}$ as 
\begin{equation}
K_{n}^{s}=L_{n}J_{n}^{s}L_{n}^{-1}=W\left( s\right) =\left( w_{ij}\left(
s\right) \right) .  \label{20}
\end{equation}%
Then 
\begin{equation}
{\scriptsize {w_{ij}\left( s\right) =\left\{ 
\begin{array}{lll}
0 & ,\;if & (-1)^{i+j}=-1 \\ 
\begin{array}{ll}
\sum\limits_{z=1}^{\frac{n}{2}}{q_{2z-1}r_{1}^{l}\alpha
_{2z-1}^{s}B_{i-1}\left( \alpha _{2z-1}\right) B_{j-1}\left( \alpha
_{2z-1}\right) } & j=1,3,\ldots ,n-1 \\ 
&  \\ 
\sum\limits_{z=1}^{\frac{n}{2}}{q_{2z}r_{2}^{l}\alpha _{2z}^{s}A_{i-1}\left(
\alpha _{2z}\right) A_{j-1}\left( \alpha _{2z}\right) } & j=2,4,\ldots ,n%
\end{array}
& ,\;if & (-1)^{i+j}=1%
\end{array}%
\right. }}  \label{21}
\end{equation}%
here $r_{1}=\sqrt{\frac{b_{1}}{c_{1}}},\;r_{2}=\sqrt{\frac{b_{2}}{c_{2}}},$%
\begin{equation*}
l=\left\{ 
\begin{array}{ll}
j-1, & \;j=1,3,5,\ldots ,n-3,n-1 \\ 
j-2, & \;j=2,4,6,\ldots ,n-2,n%
\end{array}%
,\right.
\end{equation*}
\begin{equation*}
q_{i}=\left\{ 
\begin{array}{ll}
\frac{4-\left( \frac{\alpha _{i}-a_{1}}{\sqrt{b_{1}c_{1}}}\right) ^{2}}{n+2},
& \;i=2t+1\,(t\in \mathbb{N}) \\ 
\frac{4-\left( \frac{\alpha _{i}-a_{2}}{\sqrt{b_{2}c_{2}}}\right) ^{2}}{n+2},
& \;i=2t\,(t\in \mathbb{N}).%
\end{array}%
\right.
\end{equation*}%
and $\alpha _{k}$ are the eigenvalues of the matrix $K_{n}$ $(n=2t,\,t\in 
\mathbb{N})$.

\begin{corollary}
Let $K_{n}$ be $n$-square $(n=2t,\,t\in\mathbb{N}%
;\;a_{1},a_{2},b_{1},b_{2},c_{1},c_{2}\in\mathbb{C}\setminus \left\{
0\right\})$ pentadiagonal $2$-Toeplitz matrix as in (\ref{5}), from Theorem
7 
\begin{equation}
a_1\neq 2\sqrt{b_{1}c_{1}}cos\left(\frac{(k+1)\pi}{n+2}\right)
\end{equation}
$(k=2t+1,\,t\in\mathbb{N})$ and 
\begin{equation}
a_2\neq 2\sqrt{b_{2}c_{2}}cos\left(\frac{k\pi}{n+2}\right).
\end{equation}
$(k=2t,\,t\in\mathbb{N})$. In that case, there exists the inverse and
negative integer powers of the matrix $K_{n}$.
\end{corollary}

\section{Numerical examples}

\begin{example}
Taking $n=6$ in Theorem 7, we obtain 
\begin{eqnarray*}
J_{6} &=&diag(\alpha _{1},\alpha _{2},\alpha _{3},\alpha _{4},\alpha
_{5},\alpha _{6}) \\
&=&diag(a_{1}-\sqrt{2b_{1}c_{1}},a_{2}-\sqrt{2b_{2}c_{2}},a_{1},a_{2},a_{1}+%
\sqrt{2b_{1}c_{1}},a_{2}+\sqrt{2b_{2}c_{2}})
\end{eqnarray*}%
and%
\begin{eqnarray*}
K_{6}^{s} &=&L_{6}J_{6}^{s}L_{6}^{-1}=W\left( s\right) \\
&=&\left( w_{ij}\left( s\right) \right) =\left[ 
\begin{array}{cccccc}
x_{1} & 0 & x_{7} & 0 & x_{11} & 0 \\ 
0 & x_{2} & 0 & x_{8} & 0 & x_{12} \\ 
x_{3} & 0 & x_{9} & 0 & x_{7} & 0 \\ 
0 & x_{4} & 0 & x_{10} & 0 & x_{8} \\ 
x_{5} & 0 & x_{3} & 0 & x_{1} & 0 \\ 
0 & x_{6} & 0 & x_{4} & 0 & x_{2}%
\end{array}%
\right] ,
\end{eqnarray*}%
\begin{equation*}
\left. 
\begin{array}{lll}
x_{1} & = & \frac{1}{4}\left[ \left( a_{1}+\sqrt{2b_{1}c_{1}}\right)
^{s}+2a_{1}^{s}+\left( a_{1}-\sqrt{2b_{1}c_{1}}\right) ^{s}\right] \\ 
x_{2} & = & \frac{1}{4}\left[ \left( a_{2}+\sqrt{2b_{2}c_{2}}\right)
^{s}+2a_{2}^{s}+\left( a_{2}-\sqrt{2b_{2}c_{2}}\right) ^{s}\right] \\ 
x_{3} & = & \frac{\sqrt{2}}{4}r_{1}^{-1}\left[ \left( a_{1}+\sqrt{2b_{1}c_{1}%
}\right) ^{s}-\left( a_{1}-\sqrt{2b_{1}c_{1}}\right) ^{s}\right] \\ 
x_{4} & = & \frac{\sqrt{2}}{4}r_{2}^{-1}\left[ \left( a_{2}+\sqrt{2b_{2}c_{2}%
}\right) ^{s}-\left( a_{2}-\sqrt{2b_{2}c_{2}}\right) ^{s}\right] \\ 
x_{5} & = & \frac{1}{4}r_{1}^{-2}\left[ \left( a_{1}+\sqrt{2b_{1}c_{1}}%
\right) ^{s}-2a_{1}^{s}+\left( a_{1}-\sqrt{2b_{1}c_{1}}\right) ^{s}\right]
\\ 
x_{6} & = & \frac{1}{4}r_{2}^{-2}\left[ \left( a_{2}+\sqrt{2b_{2}c_{2}}%
\right) ^{s}-2a_{2}^{s}+\left( a_{2}-\sqrt{2b_{2}c_{2}}\right) ^{s}\right]
\\ 
x_{7} & = & \frac{\sqrt{2}}{4}r_{1}\left[ \left( a_{1}+\sqrt{2b_{1}c_{1}}%
\right) ^{s}-\left( a_{1}-\sqrt{2b_{1}c_{1}}\right) ^{s}\right] \\ 
x_{8} & = & \frac{\sqrt{2}}{4}r_{2}\left[ \left( a_{2}+\sqrt{2b_{2}c_{2}}%
\right) ^{s}-\left( a_{2}-\sqrt{2b_{2}c_{2}}\right) ^{s}\right] \\ 
x_{9} & = & \frac{1}{2}\left[ \left( a_{1}+\sqrt{2b_{1}c_{1}}\right)
^{s}+\left( a_{1}-\sqrt{2b_{1}c_{1}}\right) ^{s}\right] \\ 
x_{10} & = & \frac{1}{2}\left[ \left( a_{2}+\sqrt{2b_{2}c_{2}}\right)
^{s}+\left( a_{2}-\sqrt{2b_{2}c_{2}}\right) ^{s}\right] \\ 
x_{11} & = & \frac{1}{4}r_{1}^{2}\left[ \left( a_{1}+\sqrt{2b_{1}c_{1}}%
\right) ^{s}-2a_{1}^{s}+\left( a_{1}-\sqrt{2b_{1}c_{1}}\right) ^{s}\right]
\\ 
x_{12} & = & \frac{1}{4}r_{2}^{2}\left[ \left( a_{2}+\sqrt{2b_{2}c_{2}}%
\right) ^{s}-2a_{2}^{s}+\left( a_{2}-\sqrt{2b_{2}c_{2}}\right) ^{s}\right] .%
\end{array}%
\right.
\end{equation*}
\end{example}

\begin{example}
Taking $s=3, n=8, a_{1}=1, a_{2}=i+1, b_{1}=3, b_{2}=i+3, c_{1}=5$ and $%
c_{2}=i+5 $ in Theorem 7, we obtain 
\begin{eqnarray*}
J_{8}&=&diag(\alpha_{1},\alpha_{2},\alpha_{3},\alpha_{4},\alpha_{5},%
\alpha_{6},\alpha_{7},\alpha_{8}) \\
&=&diag(-5.267,-5.280-0.668i,-1.394,-1.399-0.363i, \\
&& \quad\quad \quad\quad \quad\quad3.3394,3.399+1.637i,7.267,7.280+2.668i)
\end{eqnarray*}
and 
\begin{eqnarray*}
K_{8}^{3} &=&L_{8}J_{8}^{3}L_{8}^{-1}=W\left( 3\right) \\
&=&\left( w_{ij}\left( 3\right) \right) =\left[ 
\begin{array}{ccccc}
46 & 0 & 99 & 0 & 27 \\ 
0 & 16+68i & 0 & 62+94i & 0 \\ 
165 & 0 & 91 & 0 & 144 \\ 
0 & 118+138i & 0 & 34+134i & 0 \\ 
75 & 0 & 240 & 0 & 91 \\ 
0 & 42+102i & 0 & 180+192i & 0 \\ 
125 & 0 & 75 & 0 & 165 \\ 
0 & 110+74i & 0 & 42+102i & 0%
\end{array}
\right.
\end{eqnarray*}
\begin{eqnarray*}
\left. 
\begin{array}{cccc}
\qquad \qquad \qquad \qquad \qquad \qquad \qquad \qquad \qquad & 0 & 27 & 0
\\ 
\qquad \qquad \qquad \qquad \qquad \qquad \qquad \qquad \qquad & 6+42i & 0 & 
18+26i \\ 
\qquad \qquad \qquad \qquad \qquad \qquad \qquad \qquad \qquad & 0 & 27 & 0
\\ 
\qquad \qquad \qquad \qquad \qquad \qquad \qquad \qquad \qquad & 96+12i & 0
& 6+42i \\ 
\qquad \qquad \qquad \qquad \qquad \qquad \qquad \qquad \qquad & 0 & 99 & 0
\\ 
\qquad \qquad \qquad \qquad \qquad \qquad \qquad \qquad \qquad & 34+134i & 0
& 62+94i \\ 
\qquad \qquad \qquad \qquad \qquad \qquad \qquad \qquad \qquad & 0 & 46 & 0
\\ 
\qquad \qquad \qquad \qquad \qquad \qquad \qquad \qquad \qquad & 118+138i & 0
& 16+68i%
\end{array}
\right].
\end{eqnarray*}
\end{example}

\begin{example}
Taking $s=-4,n=10,a_{1}=1,a_{2}=2,b_{1}=3,b_{2}=4,c_{1}=5$ and $c_{2}=6$ in
Theorem 7, we get 
\begin{eqnarray*}
J_{10} &=&diag(\alpha _{1},\alpha _{2},\alpha _{3},\alpha _{4},\alpha
_{5},\alpha _{6},\alpha _{7},\alpha _{8},\alpha _{9},\alpha _{10}) \\
&=&diag(-5.7082,-6.4853,-2.8730,-2.8990,1,2,4.8730,6.8990,7.7082,10.4853)
\end{eqnarray*}%
and 
\begin{eqnarray*}
K_{10}^{-4} &=&L_{10}J_{10}^{-4}L_{10}^{-1}=W\left( -4\right) \\
&=&\left( w_{ij}\left( -4\right) \right) =\left[ 
\begin{array}{ccccc}
0.3375 & 0 & -0.0026 & 0 & -0.1999 \\ 
0 & 0.0245 & 0 & -0.0029 & 0 \\ 
-0.0043 & 0 & 0.0044 & 0 & -0.0001 \\ 
0 & -0.0043 & 0 & 0.0038 & 0 \\ 
-0.5552 & 0 & -0.0002 & 0 & 0.3337 \\ 
0 & -0.0311 & 0 & -0.0002 & 0 \\ 
0.0067 & 0 & -0.0063 & 0 & -0.0002 \\ 
0 & 0.0062 & 0 & -0.0052 & 0 \\ 
0.9148 & 0 & 0.0067 & 0 & -0.5552 \\ 
0 & 0.0388 & 0 & 0.0062 & 0%
\end{array}%
\right.
\end{eqnarray*}%
\begin{equation*}
\quad \quad \quad \quad \quad \quad \quad \quad \quad \quad \quad \quad
\left. 
\begin{array}{ccccc}
0 & 0.0015 & 0 & 0.1186 & 0 \\ 
-0.0138 & 0 & 0.0018 & 0 & 0.0077 \\ 
0 & -0.0023 & 0 & 0.0015 & 0 \\ 
-0.0001 & 0 & 0.0023 & 0 & 0.0018 \\ 
0 & -0.0001 & 0 & -0.1999 & 0 \\ 
0.0210 & 0 & -0.0001 & 0 & -0.0138 \\ 
0 & 0.0044 & 0 & -0.0026 & 0 \\ 
-0.0001 & 0 & 0.0038 & 0 & -0.0029 \\ 
0 & -0.0043 & 0 & 0.3375 & 0 \\ 
0.0311 & 0 & -0.0043 & 0 & 0.0245%
\end{array}%
\right] .
\end{equation*}
\end{example}

\end{document}